\documentclass[12pt,reqno]{amsart}
\usepackage{amsmath, amsthm, amsopn, amssymb, microtype,bm}
\usepackage{mathrsfs} 
\usepackage{mathscinet}
\usepackage{bbm}
\usepackage{enumerate, tikz, etoolbox, intcalc, geometry, caption, subcaption, afterpage}
\usepackage[english]{babel}
\usepackage{enumitem}

\usepackage[colorlinks=true,urlcolor=blue,pdfborder={0 0 0}]{hyperref}
\hypersetup{linkcolor=[rgb]{0,0,0.6}}
\hypersetup{citecolor=[rgb]{0,0.6,0}}
\usepackage[nameinlink]{cleveref}
\crefname{theorem}{Theorem}{Theorems}
\crefname{thm}{Theorem}{Theorems}
\crefname{mainthm}{Theorem}{Theorems}
\crefname{conj}{Conjecture}{Theorems}
\crefname{lemma}{Lemma}{Lemmas}
\crefname{lem}{Lemma}{Lemmas}
\crefname{remark}{Remark}{Remarks}
\crefname{prop}{Proposition}{Propositions}
\crefname{defn}{Definition}{Definitions}
\crefname{corollary}{Corollary}{Corollaries}
\crefname{cor}{Corollary}{Corollaries}
\crefname{section}{Section}{Sections}
\crefname{figure}{Figure}{Figures}
\crefname{quest}{Question}{Questions}

\newcommand{\N}{\mathbb{N}}
\newcommand{\Z}{\mathbb{Z}}

\newcommand{\R}{\mathbb{R}}
\newcommand{\B}{\mathcal{B}}

\newcommand{\kay}{\mathcal{K}}

\newcommand{\TT}{\widetilde{T}}
\newcommand{\TS}{\widetilde{S}}
\newcommand{\Tmu}{\widetilde{\mu}}

\graphicspath{ {./Pictures/} }

\newcommand{\ignore}[1]{}

\newtheorem{thm}{Theorem}[section]
\newtheorem{lemma}[thm]{Lemma}
\newtheorem{prop}[thm]{Proposition}
\newtheorem{cor}[thm]{Corollary}

\newtheorem{fact}[thm]{Fact}
\theoremstyle{definition}

\newtheorem*{thm*}{Theorem}
\newtheorem*{prop*}{Proposition}
\newtheorem*{Prob*}{Problem}

\newif\ifdraft\drafttrue

\begin{document}

\title[Counterexamples to double recurrence]{Counterexamples to double recurrence for non-commuting deterministic transformations}

\author{Zemer Kosloff}
\address{Zemer Kosloff, School of Mathematics, University of Bristol, Fry Building, Woodland Road, Bristol, BS8 1UG, United Kingdom}
\email{zemer.kosloff@bristol.ac.uk}

\address{Zemer Kosloff, Einstein Institute of Mathematics, Hebrew University of Jerusalem, Edmond J. Safra Campus, Jerusalem, 9190401, Israel.}
\email{zemer.kosloff@mail.huji.ac.il}

\author{Shrey Sanadhya}
\address{Shrey Sanadhya, Einstein Institute of Mathematics, Hebrew University of Jerusalem, Edmond J. Safra Campus, Jerusalem, 9190401, Israel.}
\email{shrey.sanadhya@mail.huji.ac.il}


\begin{abstract}
We show that if $p_1,p_2$ are injective, integer polynomials that vanish at the origin, such that either both are of degree $1$ or both are of degree $2$ or higher, then double recurrence fails for non-commuting, mixing, zero entropy transformations. This answers a question of Frantzikinakis and Host completely.
\end{abstract}

\maketitle

\section{Introduction} Furstenberg and Katznelson \cite{Furstenberg_Katznelson_1978}  obtained a multidimensional extension of  Szemerédi theorem by proving the following multiple recurrence result: Let $T_1,\ldots,T_k$ be commuting measure preserving transformations of a probability space $(X,\B,\mu)$, then for every $B \in \B$ with $\mu(B) > 0$, there exists $n\in \N$, such that
\[
\mu \left(A\cap T_1^{-n}B\cap T_2^{-n}B \cap \ldots\cap T_k^{-n}B\right) > 0.
\] 
Since then, the extension of this multiple recurrence property has been central in ergodic Ramsey theory and has had numerous applications in establishing beautiful results in infinite Ramsey theory. A  notable extension was obtained by  Bergelson and Leibman \cite{Bergelson_Leibman_1996} where it was shown that for all $p_1,\ldots,p_k$ integer polynomials which vanish at the origin, and $T_1,\ldots,T_k$ commuting measure preserving transformations, the multiple recurrence property holds along the polynomial times. 

A natural question which arises, and is the main topic of this work, is to study the recurrence problem when the systems do not necessarily commute. In this direction, Bergelson and Leibman \cite{Bergelson_Leibman_2002} showed that double recurrence holds if the two transformations $T,S$ generate a nilpotent group. 
In the non-commutative setting Furstenberg \cite[Chapter 2, page 40]{Furstenberg_1981} constructed a simple example of double non-recurrence where both the transformations have positive entropy. Recently Frantzikinakis and Host \cite{Frantzikinakis_Host_2023} showed that the failure of multiple recurrence for non-commutative positive entropy systems (as indicated by Furstenberg's example above) holds in greater generality. They proved the following. 

\begin{thm*}\cite[Proposition 1.4.]{Frantzikinakis_Host_2023} Let $(X, \mu, T)$ be an ergodic system with positive entropy, $a, b : \N \rightarrow
\Z \setminus \{0\}$ be injective sequences that miss infinitely many integers, and $F$ be an arbitrary
subset of $\N$. Then there exist a system $(X, \mu, S)$, with positive entropy, a measurable set
$A$, and $c > 0$, such that

\[
\mu\left(A\cap T^{-a(n)}A\cap S^{-b(n)}A\right) =\begin{cases}
0 &\ \text{if}\,\, n \in F,\\
c &\ \text{if}\,\, n \notin F.
\end{cases}
\] As a consequence, there exist a transformation $S$ and a set $A$ with $\mu(A)>0$ such that $\mu\left(A\cap T^{-a(n)}A\cap S^{-b(n)}A\right) = 0$ for every $n \in \N$. 
\end{thm*}

On the other hand if one of the transformation has zero entropy, it was proved (in the same paper) that double recurrence holds if one of the iterates is an integer polynomial with no constant term. 

\begin{thm*}\cite[Theorem 1.3.]{Frantzikinakis_Host_2023} Let $T, S$ be measure preserving transformations acting on a probability
space $(X, \B, \mu)$ such that the system $(X, \mu, T)$ has zero entropy. Let $p \in \Z[t]$ be a polynomial such that $p(0)=0$ and $\deg (p) \geq 2$. Then for every $A \in \B$ and $\epsilon>0$ the set
\[
\left\{n\in\mathbb{N}:\ \mu\left(A\cap T^{-n}A\cap S^{-p(n)}A\right)> \mu(A)^3 - \epsilon\right\}
\]
has positive lower density. 
\end{thm*}
We should mention that the main result of \cite{Frantzikinakis_Host_2023} is a $L^2$ convergence result for multiple ergodic averages  in the non-commuting, zero entropy setting, where one of the iterates is given by an integer polynomial (c.f. \cite[Theorem 1.1.]{Frantzikinakis_Host_2023}). The authors (\cite[Problem, page 1639]{Frantzikinakis_Host_2023}) asked whether convergence of double averages (and recurrence) holds (in this setting), if the pair of iterates $n,p(n)$ are replaced by the pair $n,n$ or by the pair $n^2,n^3$. They also asked if in general, the result hold for pairs of iterates given by arbitrary integer polynomials $p_1, p_2 \in \Z(t)$ with $p_1(0) = p_2(0) = 0$. Recently counterexamples to multiple averages for zero entropy, non commuting transformations, were constructed in \cite{Huang_Shao_Ye_2024_A}, \cite{Austin_2025}, \cite{Huang_Shao_Ye_2024_B} and \cite{kosloff_Sanadhya_2024}. Similar results were also studied in \cite{Ryzhikov_2024_A}. In this work we prove that if $T$ has zero entropy then the only case where double recurrence holds in general is \cite[Theorem 1.3.]{Frantzikinakis_Host_2023} mentioned above. 

 In a previous work \cite[Theorem 1.4.]{kosloff_Sanadhya_2024}, we showed that when the integer polynomials $p_1,p_2 \in \Z(t)$ are of a special type (high degree and large leading coefficient), then double recurrence in the non-commuting and zero entropy setting fails. The following result shows that the same is true when the degree of both polynomials is at least $2$.  

 \begin{thm}\label{thm:recurrence_1} Let $p_1,p_2:\Z\to\Z$ be injective polynomials with $p_1(0) = p_2(0) = 0$ and $\deg (p_1),\deg (p_2) \geq 2$. There exist $T,S$, two mixing, measure preserving transformations of a standard probability space $(X,\B,\mu)$, with $h_{\mu}(X,T) = h_{\mu}(X,S) = 0$ and $A\in\B$ with $\mu(A)>0$, such that for all $n\in\N$,
\[
A\cap T^{-p_1(n)}A\cap S^{-p_2(n)}A = \emptyset.
\]
\end{thm}

With this result at hand and \cite[Theorem 1.3.]{Frantzikinakis_Host_2023}, it remains to study the recurrence problem in the linear polynomial settings. Our next result shows that double recurrence fails when both polynomials are linear. 

\begin{thm}\label{thm:recurrence_2} Let $c,d \in \Z \setminus \{0\}$ then there exist $T,S$ two mixing, measure preserving transformations of a standard probability space $(X,\B,\mu)$, with $h_{\mu}(X,T) = h_{\mu}(X,S) = 0$ and $A\in\B$ with $\mu(A)>0$, such that for all $n\in\N$,
\[
A\cap T^{-cn}A\cap S^{-dn}A = \emptyset.
\]
\end{thm}

 We remark that the special case where $c=d=1$ was also studied in \cite{ryzhikov_2024_B}. 
 

 

\subsection{Structure of the paper}
While not immediately clear, it turns out that the simplest non-recurrence result (in non-commuting, zero entropy setting) is the $n,n$ case (in other words \Cref{thm:recurrence_2} for $c=d=1$). We present a simple proof in \Cref{sec:Proof_c=d=1}, as the baby case of the more difficult construction when the pair of iterates are given by injective integer polynomials of degree $2$ and higher i.e. \Cref{thm:recurrence_1}. The latter is done in \Cref{sec:Proof_rec_1}. One major new difficulty which appears in the higher degree polynomials case is that unlike the $n,n$ iterates case, the eventual failure of double recurrence (think Borel Cantelli Lemma) does not immediately imply the failure of double recurrence. As a result the proof becomes more subtle and we needed to obtain new results on the ranges of the cocycles from \cite{kosloff_Sanadhya_2024}. This part which we believe is of independent interest is carried out in \cref{sub:LCLT revisit}. 

We found the adaptation of the skew product construction in \Cref{sec:Proof_c=d=1} and \ref{sec:Proof_rec_1} to the general linear case (i.e. \Cref{thm:recurrence_2} where $c \neq 1,d \neq 1$) to be rather difficult. Hence in \Cref{sec: Gaussian} we construct Gaussian counterexamples to $n,n$ recurrence. These examples build upon the main ideas of Austin \cite{Austin_2025}, a Cartesian power trick from Rhyzikov \cite{ryzhikov_2024_B} and singular measures on the circle with polynomial Fourier decay. Then using the fact that these systems have roots of all orders we derive the general case of two linear integer polynomials which vanish at $0$. 

{\bf Acknowledgement.} The authors would like to thank Nikos Frantzikinakis for bringing the problem to our attention. The second-named author would like to thank the School of Mathematics, University of Bristol, for the hospitality. This work was partially supported by the Israel Science Foundation grant No. 1180/22.

\section{Proof of \Cref{thm:recurrence_2} for $c=d=1$}\label{sec:Proof_c=d=1} In this section, we prove counterexample to multiple recurrence in the non-commuting zero entropy setting for the pair of iterates $n,n$. This case turns out to be the most simplest one. Hence we will prove it first to illustrate the general method that will be used in the proof of \Cref{thm:recurrence_1} in \Cref{sec:Proof_rec_1}.  

An important tool that we will use is a the lattice local central limit theorem (LCLT) for deterministic systems (see \Cref{thm:LCLT_1} below) which was proved in \cite{MR4374685}. Given a probability preserving transformation $(X,\B,m,T)$, and a function $f:X\to \Z$, its \textbf{sum process} is defined by $S_n(f):=\sum_{k=0}^{n-1}f\circ T^k$, $n\in\mathbb{N}$. We will also refer to $S_n(f)$ as the cocycle corresponding to $f$. 

\begin{thm}\cite{MR4374685}\label{thm:LCLT_1}
Let $(X,\B,m,T)$ be an ergodic and aperiodic probability preserving transformation. There exists a square integrable function $f:X\to \Z$ with $\int_X fdm = 0$, such that
\[
\sup_{j\in\Z} \left|\sqrt{n}\, m\left(S_n(f)=j\right) -\frac{e^{-j^2/2n\sigma^2}}{\sqrt{2\pi \sigma^2}}\right|\xrightarrow[n\to\infty]{}0,
\]
where $\sigma^2=2 (\ln 2)^2$.
\end{thm}

Let $(Y,\mathcal{C},m,Q)$ be a mixing, zero entropy, aperiodic probability preserving transformation. Examples of such maps are the rank one mixing transformations in \cite{MR1443143,MR2652597}, the time-one map of the horocycle flow on a compact hyperbolic surface and Gaussian automorphisms with a singular spectral measure (see \cref{sec: Gaussian}). By \Cref{thm:LCLT_1}, there exists a Borel function, $f: Y \rightarrow \Z$, such that the corresponding sums process $S_n(f) : Y \rightarrow \Z$, satisfies the lattice local central limit theorem (LCLT). Sometimes by abuse of notation, we will say that $f$ satisfies LCLT. In this section, we fix such a Borel function $f$.

Let $\Sigma = \{0,1\}^{\Z}$ be the space of arrays of $\{0,1\}$. We denote by $\sigma : \Sigma \rightarrow \Sigma$ the shift, given by, 
\[
(\sigma \, \omega) (n) = \omega (n+1), \hspace{3mm} \forall n \in \Z,
\]
for $\omega\in \Sigma$. We endow $\Sigma$ with the stationary (infinite) product measure $\nu=\left(\frac{1}{2}\delta_0+\frac{1}{2}\delta_1\right)^{\Z}$ with marginals $\left(\frac{1}{2},\frac{1}{2}\right)$, and consider the full shift $(\Sigma, \B(\Sigma), \nu, \sigma)$. Here $\B(\Sigma)$ is the Borel sigma algebra generated by cylinder sets in $\Sigma$. We define $(Z,\mathcal{D},\Tmu)$ to be the Cartesian product space $Y \times \Sigma$, endowed with product measure, in other words
\[(Z,\mathcal{D},\Tmu) : = (Y \times \Sigma, \mathcal{C} \otimes \B(\Sigma), m \times \nu).\]
Let  $\TT : Y \times \Sigma \rightarrow Y \times \Sigma$ be the skew product of $Q$ and $f$ defined by 
\begin{equation}\label{eq:TT}
    \TT (y, \omega) = (Q (y), \sigma_{f(y)} (\omega)).
\end{equation}
The skew product $\TT$ is a probability preserving transformation of $(Z,\mathcal{D},\Tmu)$ and for all $n\in\N$,
\begin{equation}\label{eq:TT^n_1}
    \TT^n (y, \omega) = (Q^n(y), \sigma_{S_n(f)(y)} (\omega)).
\end{equation}

\begin{prop}\label{prop:TT_WM_0} $(Z,\mathcal{D},\Tmu,\TT)$ is a mixing, zero entropy probability preserving system.
\end{prop}
\begin{proof} We will prove that $(Z,\mathcal{D},\Tmu,\TT)$ is mixing by showing that $f$ satisfies \Cref{prop:WM} in \Cref{App_1}. It follows from  \Cref{thm:LCLT_1} that there exits $N \in \N$, such that for $j \in \Z$ and all $n > N$, we have
\[
\left|\sqrt{n}\, m\left(S_n(f)=j\right) -\frac{e^{-j^2/2n\sigma^2}}{\sqrt{2\pi \sigma^2}}\right| < \frac{1}{\sqrt{2\pi \sigma^2}}.
\]
Hence for all $n> N$, we have 
\[
m\left(S_n(f)=j\right) < \frac{e^{-j^2/2n\sigma^2}}{\sqrt{2 n \pi \sigma^2}} + \frac{1}{\sqrt{2 n\pi \sigma^2}} < \frac{2}{\sqrt{2 n\pi \sigma^2}}.
\]
Thus there exists a constant $c>0$, such that for all $n > N$ and for all $j\in\Z$,
\begin{equation}\label{eq:LCLT_bound}
m\left(S_n(f)=j\right)\leq \frac{c}{\sqrt{n}}.
\end{equation} 
Let $M\in\mathbb{N}$, then it follows,
\[
0\leq m\left(\left\|S_n(f)\right\|_\infty\leq M\right)=\sum_{|j|_\infty\leq M} m\left(S_n(f)=j\right)\leq \frac{(2M+1) c}{\sqrt{n}}\xrightarrow[n\to\infty]{}0.
\]
Hence we deduce from \Cref{prop:WM} in \Cref{App_1} that $(Z,\mathcal{D},\Tmu,\TT)$ is  mixing.  
The system $(Z,\mathcal{D},\Tmu,\TT)$ has zero entropy follows from an application of the Abramov-Rokhlin formula (for details see \cite[Proposition 2.2]{Huang_Shao_Ye_2024_A} or \cite[
Proposition 4.3]{kosloff_Sanadhya_2024}). 
\end{proof}
Let $\phi: \Sigma \rightarrow \Sigma$ be given by, 
\begin{equation*}
    (\phi \,\omega)(n) : =\begin{cases}
\omega(0) &\ \text{if}\,\, n =0,\\
1-\omega(n) &\ \text{if}\,\, n \neq 0,
\end{cases}
\end{equation*}for $\omega \in \Sigma$. We set $\Psi : Y \times \Sigma \rightarrow Y \times \Sigma$ to be $\Psi (y,\omega) = (y, \phi\,\omega)$ and define $\TS : Y \times \Sigma \rightarrow Y \times \Sigma$ to be
\begin{equation}\label{eq:TS}
    \TS  : = \Psi^{-1} \circ \TT \circ \Psi.
\end{equation}
Note that for all $n\in\N$,
\begin{equation}\label{eq:TS^n_1}
    \TS^n (y, \omega) = (Q^n(y), \phi^{-1}\circ\sigma_{S_n(f)(y)} \circ \phi(\omega)).
\end{equation}

\begin{prop}\label{prop:TS_WM_0} $(Z,\mathcal{D},\Tmu,\TS)$ is a  mixing, zero entropy probability preserving system.
\end{prop}
\begin{proof} It is easy to see that $\phi:\Sigma\to\Sigma$ is an invertible, measurable transformation which preserves $\nu$. Therefore,  $\Psi : Z \rightarrow Z$ is an invertible probability preserving transformation and $\TS$ is measure theoretically isomorphic to $\TT$, which is a mixing, zero entropy system. 
\end{proof}

\subsection{Construction of the transformations $T,S$ that satisfies \Cref{thm:recurrence_2}\label{sub: c=d=1} for $c=d=1$}\label{sec:construction_T_S} Let $(Z,\mathcal{D},\Tmu,\TT)$ and $(Z,\mathcal{D},\Tmu,\TS)$ be the  mixing, zero entropy systems as above. We set  $X = Z^{\otimes 3} : = Z \times Z \times Z $ and define two $\mu = \Tmu^{\otimes3}$ preserving systems
\[
T : X \rightarrow X, \hspace{10 mm} \mathrm{where} \hspace{2 mm} T = \TT^{\otimes3} : = \TT \times \TT \times \TT,
\]
and
\[
S : X \rightarrow X, \hspace{10 mm} \mathrm{where} \hspace{2 mm} T = \TS^{\otimes3} : = \TS \times \TS \times \TS. 
\]
We denote by $\B$ the product Borel sigma algebra $\B = \mathcal{D}^{\otimes3}$. Both the systems $(X,\B,\mu, T)$ and $(X,\B,\mu, S)$ are mixing and zero entropy. This follows from the fact that $(Z,\mathcal{D},\Tmu,\TT)$ and $(Z,\mathcal{D},\Tmu,\TS)$ are mixing and zero entropy. 

\begin{proof}[Proof of \Cref{thm:recurrence_2} for $c=d=1$] Let $B = (Y \times [1]_{0})^{\otimes3} \subset X$, then
\begin{align*}
    \mu(T^{-n}B \cap S^{-n}B) &= \left(\Tmu \left(\TT^{-n}(Y \times [1]_{0}), \TS^{-n}(Y \times [1]_{0})\right)\right)^3\\
    &= \left(\Tmu\left(\{(y,\omega) : \sigma_{S_n(f)(y)} (\omega)(0) = 1, \phi^{-1}\circ\sigma_{S_n(f)(y)} \circ \phi(\omega) (0) =1 \}\right)\right)^3.
\end{align*} 
Observe that for $\omega \in \Sigma$ we have $(\phi^{-1}\, \omega) (0) = (\omega) (0)$, hence
\begin{align*}
    \mu(T^{-n}B \cap S^{-n}B) &= \left(\Tmu\left(\{(y,\omega) : (\omega)(S_n(f)(y)) = 1, \sigma_{S_n(f)(y)} \circ \phi(\omega) (0) =1\}\right)\right)^3\\
    &= \left(\Tmu\left(\{(y,\omega) : (\omega)(S_n(f)(y)) = 1,  \phi(\omega) (S_n(f)(y)) =1\}\right)\right)^3.
\end{align*} 
Observe that if $S_n(f)(y) \neq 0$ then $\omega(S_n(f)(y)) \neq \phi(\omega) (S_n(f)(y))$, hence we have
\begin{align*}
    \mu(T^{-n}B \cap S^{-n}B) &= \left(\Tmu\left(\{(y,\omega) : S_n(f)(y)= 0, \omega(S_n(f)(y)) =1\}\right)\right)^3\\
    &= \left( \dfrac{1}{2}\cdot m\left(\{y\in Y : S_n(f)(y) = 0\}\right)  \right)^3\\
    &= \dfrac{1}{8} \cdot \left(m\left(\{y\in Y : S_n(f)(y) = 0\}\right)\right)^3.
\end{align*} 
Using the local central limit theorem (\Cref{thm:LCLT_1}) by \eqref{eq:LCLT_bound} there exists $N \in \N$ and a constant $C \in \N$ such that for all $n > N$ we have,
\begin{align*}
    \mu(T^{-n}B \cap S^{-n}B) &< \dfrac{C}{n^{3/2}}.
\end{align*} In other words, we get
\[
\sum_{n=1}^{\infty} \mu(B \cap T^{-n}B \cap S^{-n}B) < \infty.
\]
Hence using the Borel-Cantelli lemma, we get that for a.e. $x\in X$, $x \in B \cap T^{-n}B \cap S^{-n}B$ for finitely many $n \in \N$. In other words there exists $N \in \N$, such that the set  
\begin{equation}\label{eq:Def_D}
    D_N : = \{x \in B : \,\mathrm{for}\,\mathrm{all}\, n >N,\mathrm{either}\, T^n x \notin B \,\mathrm{or}\, S^n x \notin B \},
\end{equation} has positive $\mu$ measure.  We can choose $N$ to be smallest such integer and write $D=D_N$. Now let $ M\geq 0$, be the largest integer such that 
\[
A := D \cap T^{-M}D \cap S^{-M}D.
\]
is of positive $\mu$ measure. Since $D\subset B$, for all $n>N$,
\[
D \cap T^{-n}D \cap S^{-n}D\subset D \cap T^{-n}B \cap S^{-n}B=\emptyset,
\]
thus $M\leq N$. This shows that $M$ is a well defined integer and $M \in \{0,\ldots,N\}$.
In addition, as $\mu(D)>0$, $M\geq 0$.  For all $n \in \N$, we have $T^{-n} A \subset T^{-(n+M)}D$ and $S^{-n} A \subset S^{-(n+M)}D$. Hence for all $n\in\N$,
\[
A \cap T^{-n} A \cap S^{-n} A \subset D \cap T^{-(n+M)}D \cap S^{-(n+M)}D. 
\] 
By the definition of $M$, it implies that for all $n\in \N$, $A \cap T^{-n} A \cap S^{-n} A = \emptyset$ as needed. 
\end{proof}
\section{Proof of \Cref{thm:recurrence_1}}\label{sec:Proof_rec_1} The proof of \Cref{thm:recurrence_1} builds upon the ideas discussed in \Cref{sec:Proof_c=d=1}, where the case $c=d=1$ of \Cref{thm:recurrence_2} was proved. Here we will use a $2$-dimensional version of the lattice local central limit theorem (LCLT) for deterministic systems (see \Cref{thm:LCLT_2} below) which was proved in \cite{kosloff_Sanadhya_2024}. Recall that given a measure preserving system $(\Omega,\mathcal{F},m,T)$ and a function $f:\Omega\to\Z^2$, one can define a $\Z^2$ valued cocycle for $x \in \Omega$ by
\[
S_n(f)(x):=\begin{cases}
\sum_{k=0}^{n-1}f\circ T^k(x), & n\in\N,\\
(0,0),& n=0,\\
-\sum_{k=n}^{-1}f\circ T^k(x), &n\in -\N.
\end{cases}
\]
\begin{thm}\label{thm:LCLT_2}
Let $(\Omega,\mathcal{F},m,T)$ be an ergodic and aperiodic probability preserving transformation. There exists a square integrable function $f:\Omega\to\mathbb{Z}^2$ with $\int_{\Omega} fdm = 0$, such that:
\begin{enumerate}[label=(\alph*)]
    \item \cite[Theorem 1.1]{kosloff_Sanadhya_2024}\label{thm_sec_a:LCLT_2}
\[
\sup_{j\in\Z^2} n\left|m\left(S_n(f)=j\right)-\frac{1}{2\pi n\sigma^2}e^{-\frac{\|j\|^2}{2n\sigma^2}}\right|\xrightarrow[n\to\infty]{}0,
\]
where $\sigma^2=2 (\ln 2)^2$.
\vspace{0.2cm}
\item \label{thm_sec_b:LCLT_2}[See \cref{prop: 2nd property of LCLT function}] For all $N\in\N$ $$m\left(\left|\left\{S_j(f):\ -N\leq j\leq N\right\}\right|=2N+1\right)>0.$$ 
\end{enumerate}

\end{thm}

We remark that part \ref{thm_sec_b:LCLT_2} of  \Cref{thm:LCLT_2} is also a consequence of the construction of the LCLT function in \cite{kosloff_Sanadhya_2024}, however it is not proven there. In subsection \ref{sub:LCLT revisit}, we will revisit the construction of the LCLT function from \cite{kosloff_Sanadhya_2024} and prove this part.

As in \cref{sec:Proof_c=d=1}, let $(Y,\mathcal{C},m,Q)$ be a mixing, zero entropy, aperiodic probability preserving transformation. By \Cref{thm:LCLT_2} for $d=2$, there exists a Borel function, $g: Y \rightarrow \Z^2$, given by $g (y) = (g^{(1)}(y), g^{(2)}(y))$ for $y \in Y$, such that the corresponding $2$-dimensional ergodic sums process $S_n(g) : Y \rightarrow \Z^2$, given by $S_n(g) (y):=\sum_{k=0}^{n-1}g\circ Q^k (y)$,
satisfies the lattice local central limit theorem. In this subsection, we fix such a Borel function $g$. Let us define $f: Y \rightarrow \Z^2$ as 
\begin{equation}\label{eq:def_f}
    f = 2g,
\end{equation} and similarly $S_n(f) : Y \rightarrow \Z^2$ is given by $S_n(f) (y):=\sum_{k=0}^{n-1}f\circ Q^k (y)$. Observe that for $y \in Y$, 
\[\{S_n(f)(y): n \in \N \} \subset 2\Z^2.\]   

Let $\Sigma = \{0,1\}^{\Z^2}$ be the space of $2$-dimensional arrays of $\{0,1\}$. For $v \in \Z^2$, we denote by $\sigma_v : \Sigma \rightarrow \Sigma$ the $2$-dimensional shift, given by for every $\omega\in \Sigma$
\[
(\sigma_v \, \omega) (u) = \omega (u+v).
\] 
We endow $\Sigma$ with the stationary (infinite) product measure $\nu$ which assigns the measure $\frac{1}{2}$ to each of the cylinder sets $[0]_{(0,0)}$ and $[1]_{(0,0)}$ and consider the $\Z^2$ Bernoulli shift $(\Sigma, \B(\Sigma), \nu, \sigma)$. Here $\B(\Sigma)$ is the Borel sigma algebra generated by cylinder sets in $\Sigma$. We define $(Z,\mathcal{D},\Tmu)$ to be the Cartesian product space $Y \times \Sigma$, endowed with product measure, in other words
\[(Z,\mathcal{D},\Tmu) : = (Y \times \Sigma, \mathcal{C} \otimes \B(\Sigma), m \times \nu).\]
Let  $\TT : Y \times \Sigma \rightarrow Y \times \Sigma$ be the skew product of $Q$ and $f$ defined by 
\begin{equation}\label{eq:TT_1}
    \TT (y, \omega) = (Q (y), \sigma_{f(y)} (\omega)).
\end{equation}
The skew product $\TT$ is a probability preserving transformation of $(Z,\mathcal{D},\Tmu)$ and for all $n\in\N$,
\begin{equation}\label{eq:TT^n_2}
    \TT^n (y, \omega) = (Q^n(y), \sigma_{S_n(f)(y)} (\omega)).
\end{equation} Note that the transformation $(Z,\mathcal{D},\Tmu,\TT)$ defined in \eqref{eq:TT_1} is  mixing (see \Cref{prop:TT_WM_0}, where similar statement is proved for the one dimensional case) and zero entropy. The latter follows from Abramov-Rokhlin formula and using the
recurrence of the cocycle $S_n(f)$ (see \cite[
Proposition 4.3]{kosloff_Sanadhya_2024}).

For a $k\geq3$ (which will be determined later on, see \Cref{lem:A_n_k<1}), we set $X = Z^{\otimes k} : = \overbrace{Z \times Z \times\ldots \times Z}^\text{$k$ times}$ and define a  $\mu = \Tmu^{\otimes k}$ preserving system,
\[
T : X \rightarrow X, \hspace{10 mm} \mathrm{where} \hspace{2 mm} T = \TT^{\otimes k} : = \overbrace{\TT \times \TT \times\ldots \times \TT}^\text{$k$ times}.
\]
We denote by $\B$ is the product Borel sigma algebra $\B = \mathcal{D}^{\otimes k}$. Since $(Z,\mathcal{D},\Tmu,\TT)$ is  mixing and zero entropy, it follows that the system $(X,\B,\mu, T)$ is mixing and zero entropy.

\subsection{Construction of the mixing system $(X,\B,\mu, S)$ that satisfies \Cref{thm:recurrence_1}}\label{sec:construction_S} Let $f:Y \to\Z^2$ be the function as in \eqref{eq:def_f}. Let $p:\Z\to\Z$, be a polynomial, for $y \in Y$ and $N\in\N$, we define\begin{equation}\label{eq:R_N_p_y}
R_N^{(p)}(y):=\left\{S_{p(m)}(f)(y):\ 1\leq m\leq N\right\} \subset 2\Z^2.
\end{equation}
Recall that the partial sum process $S_{p(n)}(f)(y)$ is with respect to the system $(Y,\mathcal{C},m,Q)$. 
In the rest of the section, we will work with polynomials with positive leading coefficient. Note that if the statement of \cref{thm:recurrence_1} holds for polynomials with positive leading coefficient than it also holds for polynomials with negative leading coefficient by replacing $T$ or $S$ (or both) with $T^{-1}$ and $S^{-1}$ respectively. 

\begin{prop}\label{prop:Range_of_cocycle_along_p} Let $p:\Z\to\Z$ be a polynomial with a positive leading coefficient and $\deg(p)\geq 2$, then for almost every $y \in Y$,
\[
\lim_{n\to\infty} \dfrac{\left|R_{n}^{(p)}(y)\right|}{n} = 1.
\]
\end{prop}

\begin{proof} Follows from an argument similar to \cite[Proposition 4.4]{kosloff_Sanadhya_2024}. 
\end{proof}
For $y\in Y$, we define
\begin{equation}\label{eq:k_p_y}
K^{(p)}(y):= \{n\in\N : S_{p(n)}(f)(y) \neq (0,0) \,\mathrm{and}\,S_{p(n)}(f)(y) \notin R_{n-1}^{(p)}(y)\}.
\end{equation}
\begin{cor}\label{cor:k_p_y_banach} Let $p:\Z\to\Z$ be a polynomial with a positive leading coefficient and $\deg(p)\geq 2$, then for almost every $y \in Y$, $K^{(p)}(y)$ has density one.
\end{cor}
\begin{proof}
For every $z\in R_n^{(p)}(y)\setminus\{(0,0)\}$ there exists a unique $m\in K^{(p)}(y)\cap [0,n]$ such that $S_{p(m)}(f)(y) = z$. Hence it follows,  
\[
\left|\left|R_n^{(p)}(y)\right|-\left|K^{(p)}(y)\cap[0,n]\right|\right|\leq 1.
\]
Now Proposition \ref{prop:Range_of_cocycle_along_p} implies that $K^{(p)}(y)$ has density one. 
\end{proof} Let $p_1, p_2 : \Z \rightarrow \Z$ be polynomials with positive leading coefficients and $\deg(p_1), \deg(p_2) \geq 2$. For $y \in Y$, we define
\begin{equation}\label{eq:def_curly_K_y}
\mathcal{K}_y=K^{(p_1)}(y)\cap K^{(p_2)}(y).
\end{equation}
Observe that $\cref{cor:k_p_y_banach}$ implies that for almost every $y \in Y$ the set $\mathcal{K}_y$ has density one. Hence it follows that for almost every $y \in Y$, and $j\in \{1,2\}$, the set
\[
S(j,y) : = \left\{S_{p_j(n)}(f)(y):\ n\in \mathcal{K}_y\right\}\subset 2\Z^2,
\] is infinite and contains distinct terms. The fact that for almost every $y \in Y$, and $j\in \{1,2\}$, $S(j,y) \subset 2\Z^2$ implies that $\Z^2\setminus S(j,y)$ is infinite as well. We denote by $D \subset Y$ the full measure set for which the sets $S(j,y)$, $j \in \{1,2\}$ and their complements are infinite.

In what follows, for $y \in D$, we will define a permutation $\pi_y: \Z^2 \rightarrow \Z^2$, such that $\pi_y$ maps
\begin{equation}\label{eq:pi_y_maps}
    (0,0) \mapsto (0,0), \hspace{4mm} \mathrm{and} \hspace{4mm} S_{p_2(n)} (f) (y) \mapsto S_{p_1(n)} (f) (y),\, \forall n \in \mathcal{K}_y.
\end{equation} To this effect, for $y \in D$ we fix an enumeration of $\mathcal{K}_y = \{k_1(y)<k_2(y)< \ldots\}$. For ease of notation, when $y$ is known, we will denote $\mathcal{K}_y = \{k_1 < k_2 < \ldots\}$. Thus for $y \in D$ and $j\in \{1,2\}$, we enumerate $S(j,y) = \{S_{p_j(k_i)} (f) (y) \}_{i=1}^{\infty} \subset \Z^2$. For $y \in D$ and $j \in \{1,2\}$ we set 
\begin{equation}\label{eq:L_j_y}
    L(j,y) = \Z^2\setminus (S(j,y) \cup \{(0,0)\}) = \Z^2 \setminus (\{S_{p_j(k_i)} (f) (y) \}_{i=1}^{\infty} \cup \{(0,0)\}) \subset \Z^2.
\end{equation} $L(j,y)$ is also infinite as discussed above. Let $L(j,y) := \{\ell(j,y)_1, \ell(j,y)_2,\ldots\} \subset \Z^2$ be an enumeration of the set $L(j,y)$. For $y \in D$, $j \in \{1,2\}$, we have partition of $\Z^2$ of the form
\[
\Z^2 = \{(0,0)\} \cup S(j,y) \cup L(j,y) = \{(0,0)\} \cup \{S_{p_j(k_i)} (f) (y) \}_{i=1}^{\infty} \cup \{\ell(j,y)_i\}_{i=1}^{\infty}.
\] For $y\in D$, $j \in \{1,2\}$, let $\pi_{p_j,y}: \Z \rightarrow \Z^2$, be a bijective map given by,
\begin{equation}\label{eq:pi_p_j_y}
    \pi_{p_j,y}(i):=\begin{cases}
(0,0), &\ \text{for}\,\, i = 0;\\
S_{p_j(k_i)} (f)(y), &\ \text{for}\,\, i \geq 1;\\ 
\ell(j,y)_{-i}. &\ \text{for}\,\, i \leq -1.
\end{cases}
\end{equation} For $y \in D$, we define a map, $\pi_y: \Z^2 \rightarrow \Z^2$, by 
\begin{equation}\label{eq:pi_y}
    \pi_{y} = \pi_{p_1,y} \circ \pi^{-1}_{p_2,y}.
\end{equation} Note that $\pi_y$ is a permutation of $\Z^2$ and it satisfies \eqref{eq:pi_y_maps} as needed. For $y\in D$, let $\pi_y$ be as above, we set $\Psi_{\pi_y}: \Sigma \rightarrow \Sigma$ to be the map
\begin{equation}\label{eq:Psi_pi_y}
    \Psi_{\pi_y} (\omega)(i,j):=\begin{cases}
\omega(0,0), &\ \text{for}\,\, (i,j) = (0,0);\\
1-\omega(\pi_y(i,j)) = 1-\omega(S_{p_1(n)} (f)(y)), &\ \text{for}\, (i,j) = S_{p_2(n)} (f)(y), n \in \mathcal{K}_y ;\\
\omega(\pi_y(i,j)), &\  \text{otherwise}.
\end{cases}
\end{equation} Recall we defined $(Z,\mathcal{D},\Tmu) = (Y \times \Sigma, \mathcal{C} \otimes \B(\Sigma), m \times \nu)$. Let $\widetilde{R}: Z \rightarrow Z$ be given by,
\begin{equation}\label{eq:def_R}
   \widetilde{R}(y,\omega) :=\begin{cases}
(y, \Psi_{\pi_y} \omega ), &\ \text{for}\,\, y \in D;\\
(y,\omega), &\ \text{for}\,\, y \in Y\setminus D.
\end{cases}
\end{equation} We finally define, $\TS : Z \rightarrow Z$ as, 
\begin{equation}\label{def:S}
    \TS : = \widetilde{R}^{-1} \circ \TT \circ \widetilde{R}.
\end{equation}
Observe that for $n \in \N$, $y \in D \cap Q^{-n} D$ and $\omega \in \Sigma$, we have
\begin{equation}\label{eq:TS^n_2}
    \TS^n(y, \omega) = (Q^n(y), (\Psi^{-1}_{\pi_{Q^n(y)}} \circ \sigma_{S_n(f)(y)} \circ \Psi_{\pi_y}) (\omega )).
\end{equation}
It follows from an argument similar to \cite[Proposition 4.10]{kosloff_Sanadhya_2024} that the map $\widetilde{R}$ defined in \eqref{eq:def_R} is an invertible measure preserving transformation. Thus $(Z,\mathcal{D},\Tmu,\TS)$ is a  mixing probability preserving system with zero entropy. 

Recall that for $k\geq3$, we set $X = Z^{\otimes k} : = \overbrace{Z \times Z \times\ldots \times Z}^\text{$k$ times}$. We define a $\mu=\tilde{\mu}^{\otimes k}$ preserving system
\[
S : X \rightarrow X, \hspace{10 mm} \mathrm{where} \hspace{2 mm} S = \TS^{\otimes k} : = \overbrace{\TS \times \TS \times\ldots \times \TS}^\text{$k$ times}.
\]
Recall $\B$ is the product Borel sigma algebra $\B = \mathcal{D}^{\otimes k}$. Since $(Z,\mathcal{D},\Tmu,\TS)$ is  mixing and zero entropy, it follows that the system $(X,\mathcal{B},\mu, S)$ is mixing and zero entropy.

Let $x = (y_\ell,\omega_\ell)_{\ell=1}^k \in X = Z^{\otimes k}$, where  $y_\ell \in Y$ and $\omega_\ell \in \Sigma$ for $\ell \in \{1,\ldots,k\}$. For convenience, we will also write 
\[
x = (\overline{y},\overline{\omega}),
\] where $\overline{y} = (y_\ell)_{\ell =1}^k \in Y^{\otimes k}$ and $\overline{\omega} = (\omega_\ell)_{\ell =1}^k \in \Sigma^{\otimes k}$. Recall for $y_\ell \in Y$ and polynomials $p_1,p_2 : \Z \rightarrow \Z$, we defined $\kay_{y_\ell}=K^{(p_1)}(y_\ell)\cap K^{(p_2)}(y_{\ell})$ (see \eqref{eq:def_curly_K_y}), where $K^{(p_i)}(y_\ell)$ are as in \eqref{eq:k_p_y}. 
In what follows for brevity we denote,
\[\kay_{y_\ell} = \kay_\ell,\]
where $\ell \in \{1,\ldots,k\}$. For $n \in \N$ and $\ell \in \{1,\ldots,k\}$,
\begin{equation}\label{def:A_n_ell}
    A_n(\ell) : = \{\overline{y} \in Y^{\otimes l} : \exists t \in \{1,\ldots,\ell\}, n \in \kay_t\}.
\end{equation} In other words for $n \in \N$ and $\ell \in \{1,\ldots,k\}$, 
\begin{equation}\label{def:A_n_ell*}
    A_n(\ell) : = \left\{\overline{y} \in Y^{\otimes l} : \bigcup_{t=1}^\ell \{n \in \kay_t\}\right\}.
\end{equation}

\begin{lemma}\label{lem:A_n_k<1}

Let $p_1,p_2:\Z\to\Z$ be injective polynomials with positive leading coefficient, $p_1(0) = p_2(0) = 0$ and $\deg (p_1),\deg (p_2) \geq 2$. There exists $k \geq 3$ such that
\[
\sum_{n=1}^{\infty} m^{\otimes k}(Y^{\otimes k} \setminus A_n(k))<1. 
\]
\end{lemma}

\begin{proof} Let $\ell \in \N$ and $n\in \N$. The set $Y^{\otimes (\ell+1)} \setminus A_n(\ell)$ is an intersection of $\ell$ independent events, each of which if of measure $m(Y\setminus A_n(1))$. Consequently, for every $\ell, n\in \N$,
\begin{equation}\label{eq:product_A_n}
    m^{\otimes \ell}(Y^{\otimes \ell} \setminus A_n(\ell)) =\left(m(Y\setminus A_n(1))\right)^\ell.
\end{equation}
It is shown in \cite[page 21]{kosloff_Sanadhya_2024} that there exists $M \in \N$ such that for $n > M$, we have
\begin{equation}\label{eq:Bound_An3}
   m(Y \setminus A_n(1)) \leq \dfrac{\pi c}{\sqrt{n}},
\end{equation}
for some constant $c \in \N$. Hence 
\begin{equation}\label{eq:A_n_3}
\sum_{n=1}^{\infty} m^{\otimes 3}(Y^{\otimes 3} \setminus A_n(3))< (\pi c)^3 \sum_{n=1}^{\infty} \dfrac{1}{n^{3/2}} < \infty. 
\end{equation}
Equation \eqref{eq:product_A_n} implies that for all $\ell\geq 3$, 
\begin{equation}\label{eq:A_n_ell_+1} 
 m^{\otimes \ell}(Y^{\otimes \ell} \setminus A_n(\ell)) \leq  m^{\otimes 3}(Y^{\otimes 3} \setminus A_n(3)).     
\end{equation}
It follows from \eqref{eq:A_n_3} and \eqref{eq:A_n_ell_+1} that there exists a $N \in \N$, such that for all $k\geq 3$,
\[
\sum_{n=N}^{\infty} m^{\otimes k}(Y^{\otimes k} \setminus A_n(k)) < \dfrac{1}{2}. 
\] Thus we have,
\[
\sum_{n=1}^{\infty} m^{\otimes k}(Y^{\otimes k} \setminus A_n(k)) < \sum_{n=1}^{N-1} m^{\otimes k}(Y^{\otimes k} \setminus A_n(k)) + \frac{1}{2}.
\] 
It remains to show that for some large enough $k$, 
\[
\sum_{n=1}^{N-1} m^{\otimes k}(Y^{\otimes k} \setminus A_n(k))\leq \frac{1}{2}.
\]
The first step towards this goal is to show that for all $n\in \N$, $m(Y\setminus A_n(1)) < 1$. To see this observe that for $n \in \N$, 
\begin{equation}\label{eq: ooooffff}
m(Y\setminus A_n(1)) = m \left(\left\{S_{p_1(n)}(f)\in R_{n-1}^{\left(p_1\right)}\cup \{(0,0)\}\right\}\bigcap \left\{S_{p_2(n)}(f)\in R_{n-1}^{\left(p_2\right)}\cup\{(0,0)\}\right\}\right). 
\end{equation}
Define $M=\max\{\left|p_i(k)\right|:\ i\in\{1,2\},\ 1\leq k\leq n\}$ and note that if one of the two conditions in the right hand side happen then 
\[
\left|\left\{S_j(f):\ -M\leq j\leq M\right\}\right|<2M+1.
\]
We deduce from this and \cref{thm:LCLT_2}.\ref{thm_sec_b:LCLT_2} that for every $n \in \N$,
\begin{align*}
m(Y\setminus A_n(1))&\leq m\left(\left|\left\{S_j(f):\ -M\leq j\leq M\right\}\right|<2M+1\right)\\
&=1-m\left(\left|\left\{S_j(f):\ -M\leq j\leq M\right\}\right|=2M+1\right)<1.
\end{align*}
This shows that  there exists $a<1$ such that for all $1\leq n\leq N-1$, 
\[
m(Y\setminus A_n(1)) < a
\]
We conclude that there exists a large enough $k \in \N$, such that
\begin{align*}
\sum_{n=1}^{N-1} m^{\otimes k}(Y^{\otimes k} \setminus A_n(k))&=\sum_{n=1}^{N-1} m(Y \setminus A_n(1))^k\\
&\leq Na^k< \frac{1}{2},
\end{align*}
which proves the claim. 
\end{proof}

\begin{prop}\label{prop:m^k(barC)>0}
There exists $k\geq 3$ such that
\[
m^{\otimes k} \left(\left\{\overline{y} \in Y^{\otimes k} : \bigcap_{n=1}^{\infty} \bigcup_{t=1}^k \{n \in \kay_t\}\right\}\right) >0.
\]   
\end{prop} 

\begin{proof} For a set $B \subset Y^{\otimes k}$, we denote by $B^C = Y^{\otimes k} \setminus B$. Note that
\[
\left\{\overline{y} : \bigcap_{n=1}^{\infty} \bigcup_{t=1}^k \{n \in \kay_t\}\right\}^C = \left\{\overline{y} : \bigcup_{n=1}^{\infty} \bigcap_{t=1}^k \{n \notin \kay_t\}\right\} = \left\{\overline{y} : \bigcup_{n=1}^{\infty} (Y^{\otimes k} \setminus A_n(k))\right\}. 
\]
Since
\[
m^{\otimes k} \left(\left\{\overline{y} : \bigcup_{n=1}^{\infty} (Y^{\otimes k} \setminus A_n(k))\right\}\right) < \sum_{n=1}^{\infty} m^{\otimes k}(Y^{\otimes k} \setminus A_n(k)),
\]
the claim follows from \Cref{lem:A_n_k<1}.
\end{proof}
\begin{proof}[Proof of \Cref{thm:recurrence_1}] Recall that the set $D \subset Y$ is the full measure set such that for $y \in Y$ the permutation $\pi_y: \Z^2 \rightarrow \Z^2$ (see \eqref{eq:pi_y}) is well defined. Let $(y,\omega) \in D \times \Sigma$, we claim that if $n \in \mathcal{K}_y$, (see \eqref{eq:def_curly_K_y}) then either $\TT^{p_1(n)} (y,\omega) \notin D \times [0]_{0,0} $ or $\TS^{p_2(n)} (y,\omega) \notin D \times [0]_{0,0}$. To see this, note that by \eqref{eq:TS^n_2}, $\TS^{p_2(n)} (y,\omega) \in D \times [0]_{0,0}$ implies that
\[
(\Psi^{-1}_{\pi_{Q^n(y)}} \circ \sigma_{S_{p_2(n)}(f)(y)} \circ \Psi_{\pi_y}) (\omega ) (0,0) = 0.
\]
If $(y',\omega') \in D \times \Sigma$, then $(\Psi^{-1}_{\pi_{y'}}\omega') (0,0) = \omega' (0,0)$. Thus we get,
\[
( \sigma_{S_{p_2(n)}(f)(y)} \circ \Psi_{\pi_y}) (\omega ) (0,0) = 0.
\] In other words
\[
(\Psi_{\pi_y} \omega ) (S_{p_2(n)}(f)(y)) = 0.
\] 
Finally \eqref{eq:Psi_pi_y} implies that for $(y,\omega) \in D \times \Sigma$ and $n\in \mathcal{K}_{y}$,
\begin{equation}\label{eq:contra}
   (\Psi_{\pi_y} \omega ) (S_{p_2(n)}(f)(y))=
1- \omega(S_{p_1(n)}(f)(y)) = 0. 
\end{equation}
On the other hand if $\TT^{p_1(n)} (y,\omega) \in D \times [0]_{0,0} $ then 
\[
\omega(S_{p_1(n)}(f)(y))=\sigma_{S_{p_1(n)}(f)(y)} (\omega) (0,0) = 0.
\]
(see \eqref{eq:TT^n_2}) which contradicts \eqref{eq:contra}. Now we define $\overline{C} \subset Y^{\otimes k}$ as follows :
\[
\overline{C} : = \left\{\overline{y} \in Y^{\otimes k} :\,\mathrm{for}\,\,\mathrm{all}\,\,n\in\N,\, \,\mathrm{there}\,\,\mathrm{exists}\, t \in \{1,\ldots,k\},\, \,\mathrm{such}\,\mathrm{that}\,\, n \in \kay_{t}\right\} \cap D^{\otimes k}. 
\]
Since $D^{\otimes k} \subset Y^{\otimes k}$ is a full $m^{\otimes k}$-measure set, it follows from \Cref{prop:m^k(barC)>0} that $m^{\otimes k}(\overline{C})>0$. Let $A = \overline{C} \times [0]_{(0,0)}^{\otimes k} \subset X$ and $B = D^{\otimes k} \times [0]_{(0,0)}^{\otimes k} \subset X$. Then $\mu(A) > 0$ and $A \subset B$. Observe that for all $n \in \N$,
\begin{align*}
    &A \cap T^{-p_1(n)}A\cap S^{-p_2(n)}A =  A \cap T^{-p_1(n)}B\cap S^{-p_2(n)}B\\
    & \subset \left\{ (y_\ell,\omega_\ell)_{\ell=1}^k \in X :  \bigcup_{t=1}^k \{n \in \kay_t \,\, \mathrm{and}\,\, \TT^{p_1(n)}(y_t,\omega_t), \TS^{p_1(n)}(y_t,\omega_t) \in D \times [0]_{0,0} \} \right\} = \emptyset.
\end{align*}
    
\end{proof}

\subsection{Proof of \Cref{thm:LCLT_2} part \ref{thm_sec_b:LCLT_2}}\label{sub:LCLT revisit}

For $k \in \N$, let
\[
p_k:=\begin{cases}
2^k, & k\ \text{even},\\
2^k+1, &\ k\ \text{odd},
\end{cases}
\] and
$d_k:=2^{k^2}$. Similarly, let $\alpha_1=\frac{1}{2}$ and $\alpha_k:=\frac{1}{p_k\sqrt{k\log(k)}}$ for $k\geq 2$.

The construction of a $2$-dimensional LCLT function in \cite{kosloff_Sanadhya_2024} relies on the following proposition.

\begin{prop}\cite[Proposition 2.1. for $D=2$]{kosloff_Sanadhya_2024}\label{prop:functions_defn}
Let $(\Omega,\mathcal{F},m,T)$ be an ergodic and aperiodic probability preserving system. There exists $\bar{f}_k^{(i)}:\Omega\to \{-1,0,1\}$, $k\in\mathbb{N},i\in\{1,2\}$ such that:
\begin{enumerate}[label=(\alph*)]
\item\label{prop_sec_a:functions_defn} For all $k\in\mathbb{N}$ and $i\in\{1,2\}$, 
\[
m\left(\bar{f}_k^{(i)}=1\right)=m\left(\bar{f}_k^{(i)}=-1\right)=\frac{\alpha_k^2}{2}.
\]
\item\label{prop_sec_b:functions_defn} For every $k\in\N$, the functions $\left\{\bar{f}_k^{(i)}\circ T^j:\ 0\leq j\leq 2d_k+p_k,\ i\in\{1,2\}\right\}$ are i.i.d. 
\item\label{prop_sec_c:functions_defn} For every $k\in\N$, the functions $\left\{\bar{f}_k^{(i)}\circ T^j:\ 0\leq j\leq 2d_k+p_k,\ 1\leq i\in\{1,2\}\right\}$ are independent of 
\[
\mathcal{A}_k=\left\{\bar{f}_l^{(i)}\circ T^j:\ 1\leq l<k, i\in\{1,2\},\ 0\leq j\leq 2d_k+p_k\right\}. 
 \]
\end{enumerate}
\end{prop}
With this proposition at hand, for $i\in \{1,2\}$, we define
\[
f^{(i)}_k:=\sum_{j=0}^{p_k-1}\left(\bar{f}^{(i)}_k\circ T^j-\bar{f}^{(i)}_k\circ T^{d_k+j}\right).
\]
We denote by $f^{(i)}:=\sum_{k=0}^\infty f^{(i)}_k$ and set $f_k:=\left(f_k^{(1)},f_k^{(2)}\right):\Omega \to \Z^2$. Finally we denote 
\[
f:=\sum_{k=0}^\infty f_k = \left(\sum_{k=0}^\infty f^{(1)}_k, \sum_{k=0}^\infty f^{(2)}_k \right) = (f^{(1)},f^{(2)}).
\]

\begin{thm*}\cite[Theorem 2.4.]{kosloff_Sanadhya_2024} The function $f:\Omega \to\mathbb{Z}^2$ is well defined and it satisfies the local limit theorem with $\sigma^2 := 2(\mathrm{ln}\, 2)^2 $. That is,
\[
\sup_{j\in\Z^2} n\left|m\left(S_n(f)=j\right)-\frac{1}{(2\pi n\sigma^2)}e^{-\frac{\|j\|^2}{2n\sigma^2}}\right|\xrightarrow[n\to\infty]{}0.
\]
\end{thm*} 
This statement gives part \ref{thm_sec_a:LCLT_2} of \cref{thm:LCLT_2}.
\begin{prop}\label{prop: 2nd property of LCLT function}
The function $f:\Omega \to\mathbb{Z}^2$  satisfies for all $N\in\N$, 
$$m\left(\left|\left\{S_{j+N}(f)-S_N(f):\ -N\leq j\leq N\right\}\right|=2N+1\right)>0.$$ 
\end{prop}
 Part \ref{thm_sec_b:LCLT_2} follows from \Cref{prop: 2nd property of LCLT function}, as shown in \Cref{cor: 2nd property of LCLT function} below. 
\begin{cor} \label{cor: 2nd property of LCLT function}
The function $f:\Omega \to\mathbb{Z}^2$  satisfies for all $N\in\N$, 
\[
m\left(\left\{S_n(f):\ -N\leq n\leq N\right\}=2N+1\right)>0. 
\]
\end{cor}

\begin{proof}
Fix $N\in\N$. First note that as $T$ preserves $m$, 
\[
m\left(\left\{S_n(f):\ -N\leq n\leq N\right\}=2N+1\right)=m\left(\left\{S_n(f)\circ T^N:\ -N\leq n\leq N\right\}=2N+1\right).
\]
Secondly, 
\[
S_n(f)\circ T^N=\begin{cases}
\sum_{k=0}^{n-1}f\circ T^{k+N},& 1\leq n\leq N,\\
(0,0), &n=0,\\
-\sum_{k=n}^{-1}f\circ T^{k+N},& -N\leq n\leq 1.
\end{cases}=\begin{cases}
\sum_{k=N}^{n+N-1}f\circ T^k,& 1\leq n\leq N,\\
(0,0), &n=0,\\
-\sum_{k=n+N}^{N-1}f\circ T^k,& -N\leq n\leq 1.
\end{cases}
\]
We deduce that for all $-N\leq n\leq N$,
\[
S_n(f)\circ T^N=S_{n+N}(f)-S_N(f).
\]
The claim now follows from \cref{prop: 2nd property of LCLT function}. 
\end{proof}
Endow $\Z^2$  with the lexicographic ordering which we will denote by $\leq$. The proof of \cref{prop: 2nd property of LCLT function}, goes by showing that the the event 
\begin{equation}\label{eq: Goal Event}
(0,0)<S_1(f)<\cdots<S_{2N}(f),
\end{equation}
has positive $m$ measure. Note that \eqref{eq: Goal Event} implies that for all $-N\leq n<m\leq N$,
\[
\left(S_{N+m}(f)-S_N(f)\right)-\left(S_{N+n}(f)-S_N(f)\right)=\sum_{l=n}^{m-1}\left(S_{N+l+1}(f)-S_{N+l}(f)\right)\neq (0,0),
\]
and in particular 
\[
\left|\left\{S_{j+N}(f)-S_N(f):\ -N\leq j\leq N\right\}\right|=2N+1.
\]
In the much simpler toy model of the  the simple random walk on $\Z^2$, the event that all the first $2N$ moves were $(1,0)$ implies \eqref{eq: Goal Event}; the latter event has probability $4^{-2N}$ and thus the probability of \eqref{eq: Goal Event} is at least $4^{-2N}$. The proof of \cref{prop: 2nd property of LCLT function}, goes by carefully studying the layers of independence in the definition of $f$ in order to show that, the event that for all $1\leq j\leq N$, the first coordinate of $S_{j+1}(f)$ is strictly larger than the one of $S_j(f)$, is a positive measure set.

Now to specifics. Let $N\in\N$ be fixed. For all $0\leq n<2N$, we decompose $$S_{n+1}\left(f^{(1)}\right)-S_{n}\left(f^{(1)}\right)=Z(n)+Y(n),$$ where 
\begin{align*}
Z(n)&:=\sum_{k: 2N\geq p_k}f_k^{(1)}\circ T^n,\\
Y(n)&:=\sum_{k: 2N<p_k} f_k^{(1)}\circ T^n.
\end{align*}

For each $0\leq n\leq 2N$, $Y(n)$ is a finite sum of bounded random variables. The following is then obvious. 

\begin{fact}\label{fact: Z does not matter}
There exists $\Z\ni M<0$ such that for all $x\in \Omega$, $\min_{0\leq n< 2N}Z(n)(x)\geq M$. 
\end{fact}

Using \Cref{prop:functions_defn}.\ref{prop_sec_c:functions_defn} and the definition of $f_k,\ k\in\N$, we will show that the functions $Z(n)$ and $Y(n)$ are independent (see proof of \Cref{prop: 2nd property of LCLT function} below). The following Lemma is the main new input in the proof of \Cref{prop: 2nd property of LCLT function}. 

\begin{lemma}\label{lem: main new input in LCLT}
For every $N\in\N$ and $C\in\N$, 
\[
m\left(\forall n\in[0,2N),\ Y(n)>C\right)>0. 
\]
\end{lemma}
We first prove \cref{prop: 2nd property of LCLT function}, assuming \Cref{lem: main new input in LCLT}, and then prove the Lemma. 
\begin{proof}[Proof of \cref{prop: 2nd property of LCLT function}]
Fix $N\in \N$. Let $K=K(N)\in\N$ be the minimal positive integer such that $2N<p_K$. By definition, $Z(n)$ is a (linear) function of 
\[
\left\{\bar{f}^{(1)}_k\circ T^j:\ 1\leq k< K(n), \  0\leq j\leq 2d_K+p_K\right\}.
\]
Similarly $Y(n)$ is a function of 
\[
\left\{\bar{f}^{(1)}_k\circ T^j:\ k\geq K(n), \  0\leq j\leq 2d_k+p_k\right\}.
\]
Thus by \cref{prop:functions_defn}.\ref{prop_sec_c:functions_defn} and the definition of $f_k,\ k\in\N$ it follows that, $Z(n)$ and $Y(n)$ are independent. 

Now let $M<0$ be as in \Cref{fact: Z does not matter} and write,
\[
A:=\left\{x\in \Omega:\ \forall n\in[0,2N),\ Y(n)>-M\right\}. 
\]
For all $x\in A$ the following holds: for all $0\leq n <2N$,
\begin{align*}
S_{n+1}\left(f^{(1)}\right)(x)-S_n\left(f^{(1)}\right)(x)&=Y(n)(x)+Z(n)(x)\\
&> -M+\min_{0\leq n<2N}Z(n)(x)> 0. 
\end{align*}
Consequently for all $x\in A$, 
\[
(0,0)< S_1(f)(x)<\ldots <S_{2N}(f)(x)
\]
and as a consequence (see the discussion following \eqref{eq: Goal Event})
\[
m\left(\left|\left\{S_{N+j}(f)-S_N(f):-N\leq j\leq N\right\}\right|=2N+1\right)\geq m(A).
\]
By \cref{lem: main new input in LCLT} with $C = -M$, we have $m(A)>0$ and we are done. 
\end{proof}
\begin{proof}[Proof of \cref{lem: main new input in LCLT}]

Fix $C\in\N$. Let $\kappa=\kappa(N)$ be the smallest positive integer such that $2N<p_\kappa$ and $K=K(N)$ the smallest integer largest than $\kappa$ which satisfies in addition that $2p_K<d_K$. Note that the latter condition holds since $\lim_{k\to\infty}\frac{p_k}{d_k}=0$.  We will now define a set $D\in\B$ which we will show has positive measure and is contained in the required set $\{\min_{0\leq n<2N}Y(n)>C\}$. 

The set $D$ is defined via $D=\cap_{k=\kappa(N)}^\infty D_k$ where the sets $D_k$ are defined by the following rules: 
\begin{itemize}
    \item For all $K(N)\leq k<K(N)+C$ define,
\[
D_k:=\left\{x\in\Omega :\  \text{for all}\ 0\leq j< p_k+2N,\ \bar{f}_k^{(1)}\circ T^j(x)=1\ \text{and}\ \bar{f}_k^{(1)}\circ T^{d_k+j}(x)=0\right\}.
\]
Note that as $p_k+2N<2p_k$ there is no contradiction in the definition of $D_k$.
\vspace{3mm}
\item For $k\geq \kappa(N)$ which is not in $[K(N),K(N)+C)$ we define
\[
D_k:=\left\{x\in\Omega :\  \text{for all}\ 0\leq j< p_k+2N,\ \bar{f}_k^{(1)}\circ T^j(x)=\bar{f}_k^{(1)}\circ T^{d_k+j}(x)=0 \right\}.
\]
\end{itemize}

We first show that $D\subset \{\min_{0\leq n<2N}Y(n)>C\}$. Indeed, let $0\leq n < 2N$. First notice that for all $k\geq \kappa(N)$ which is not in $[K(N),K(N)+C)$ and $x\in D_k$, 
\[
f_k\circ T^{n}(x)=\sum_{j=n}^{p_k+n-1}\left(\bar{f}_k^{(1)}\circ T^{j}(x)-\bar{f}_k^{(1)}\circ T^{d_k+j}(x)\right)=0,
\]
as all the functions in the sum are $0$ on $D_k$. 

Similarly, for $K(N)\leq k<K(N)+C$ and $x\in D_k$,
\begin{align*}
f_k\circ T^{n}(x)&=\sum_{j=n}^{p_k+n-1}\left(\bar{f}_k^{(1)}\circ T^{j}(x)-\bar{f}_k^{(1)}\circ T^{d_k+j}(x)\right)\\
&=\sum_{j=n}^{p_k+n-1}\bar{f}_k^{(1)}\circ T^{j}(x)=p_k>1.
\end{align*}
These two observations together show that for all $x\in \cap_{k\geq \kappa(N)} D_k$,
\[
Y(n)=\sum_{k\geq \kappa(n)}f_k^{(1)}\circ T^n(x)=\sum_{k=K(N)}^{K(N)+C-1} f_k^{(1)}\circ T^n(x)>C.
\]
This shows that for all $0\leq n<2N$ and $x\in D$, $Y(n)>C$. 

It remains to show that $m(D)>0$. By \cref{prop:functions_defn}. \ref{prop_sec_b:functions_defn}, each $D_k$ is a finite intersection of independent, positive measure sets and thus for all $k\geq \kappa(N)$, 
\[
m\left(D_k\right)>0. 
\]
Furthermore, for all $k\geq K(N)+C$,
\begin{align*}
m\left(D_k\right)&=\prod_{j=0}^{p_k+2N-1}\left(m\left(\bar{f}_k^{(1)}\circ T^{j}=0\right)\cdot m\left(\bar{f}_k^{(1)}\circ T^{j+d_k}=0\right)\right)\\
&=\left(1-\alpha_k^2\right)^{p_k+2N}\\
&\geq \exp\left(-\alpha_k^2(p_k+2N)\right)\geq \exp\left(-\frac{2}{p_k}\right).
\end{align*}
Here the last inequality follows from $p_k+2N<2p_k$ and $p_k\alpha_k^2\leq\frac{1}{p_k}$. Finally, by \cref{prop:functions_defn}. \ref{prop_sec_c:functions_defn}, $\left\{D_k\right\}_{k\geq \kappa(N)}$ are independent. We conclude that,
\begin{align*}
m(D)&=\prod_{k\geq \kappa(N)}m\left(D_k\right)\\
&=\left(\prod_{k= \kappa(N)}^{K(N)+C-1} m\left(D_k\right)\right)\left(\prod_{k\geq K(N)+C} m\left(D_k\right)\right)\\
&\geq \left(\prod_{k= \kappa(N)}^{K(N)+C-1} m\left(D_k\right)\right)\left(\prod_{k\geq K(N)+C} \exp\left(-\frac{2}{p_k}\right)\right).
\end{align*}
The first term in the product is positive as a finite product of positive numbers. The second is positive since $\sum_{k=1}^\infty\frac{1}{p_k}<\infty$. It follows that $m(D)>0$, and the proof is concluded. 
\end{proof}

\section{Gaussian systems proof of \cref{thm:recurrence_2}}\label{sec: Gaussian}

A real valued stationary sequence of random variables $\left(X_n\right)_{n=-\infty}^\infty$ is a Gaussian process if for every $n_1<n_2<\cdots<n_m$ and $c_1,\ldots c_m\in\R$, $\sum_{k=1}^mc_kX_{n_k}$ is a Gaussian random variable. If in addition the expectation of $X_0$ is $0$, we will refer to the process as a stationary symmetric Gaussian process. The covariance function $\rho:\Z\to\R$ of the (symmetric) process defined by $\rho(n)=\mathbb{E}\left(X_0X_n\right)$ is a positive definite function. By Bochner's Theorem there exists $\sigma$ a symmetric measure on the circle such that for all $n\in\Z$,
\[
\mathbb{E}\left(X_0X_n\right)=\hat{\sigma}(n)=\int_{S^1}e^{int}d\sigma. 
\]
In this case $\sigma$ is called the spectral measure of the process. In the other direction, for every symmetric probability measure on the circle $\sigma$ there exists a symmetric Gaussian process whose spectral measure is $\sigma$.

A probability preserving system $(X,\B,\mu,T)$ is a Gaussian automorphism with spectral measure $\sigma$ if there exists $f\in L^2(X,\mu)$ such that
\begin{itemize}
\item $\int fd\mu=0$. 
\item The process $\left\{f\circ T^n\right\}_{n\in\Z}$ is Gaussian and for all $n\in\Z$, $\int f \cdot f\circ T^nd\mu=\hat{\sigma}(n)$,
\item the  
 $\sigma$-algebra generated by the functions $\left\{f\circ T^n\right\}_{n\in\Z}$ is $\B \mod \mu$. 
\end{itemize}

In the course of this section, we will make use of the following properties of Gaussian automorphisms. The proofs for these statements (unless specific references are mentioned) are mentioned in \cite{MR832433} or \cite{Janvresse_Roy_DeLaRue_2023}. 
\begin{enumerate}[label=(\alph*)]
    \item $T$ is mixing if and only if $\sigma$ is a Rajchman measure meaning that $\hat{\sigma}(n)\xrightarrow[|n|\to\infty]{}0$.
    \item If $\sigma$ is a singular measure (with respect to Lebesgue) then $(X,\B,\mu,T)$ is a zero entropy transformation \cite{MR191718,MR1231420}.
    \item $T$ can be embedded in a flow, in other words, there exists a $\mu$ preserving $\R$ - flow $\left\{\phi_s\right\}_{s\in\R}$ such that $\phi_1=T$. As a consequence for every $d\in \Z\setminus\{0\}$,  $R=\phi_{1/d}$ satisfies $R^d=T$. If $T$ is ergodic (respectively mixing) then so is $R$. 
    \item Let $U:L^2(X,m)\to L^2(X,m)$ be the Koopman  operator of $T$ and $\mathcal{H}$  the $L^2$ closure of the linear span of $\left\{U^nf:\ n\in\Z\right\}$. Every function in $\mathcal{H}$ is a Gaussian random variable and $\mathcal{H}$ is $U$ invariant (see \cite[Appendix B-D]{Kechris_2010}). 
    \item The space $L^2(X,m)$ is isomorphic to the Fock space over $\mathcal{H}$ (see \cite[Section 4.2 and 4.3]{Janvresse_Roy_DeLaRue_2023} and \cite[Appendix B-D]{Kechris_2010}). 
    \item Given a unitary operator $V: \mathcal{H}\to\mathcal{H}$, there is a canonical extension, called the second quantization of $V$, which is a unitary operator on the whole of $L^2(X,\mu)$ (see \cite[Proposition 4.1]{Janvresse_Roy_DeLaRue_2023}).
    \item If $V$ restricted to $\mathcal{H}$ is unitarily equivalent to $U$ restricted to $\mathcal{H}$ then there exists $S$, a $(X,\B,\mu)$ measure preserving system which is (measure theoretically) isomorphic to $T$ such that for all $n\in \Z$, $V^nf=f\circ S^n$. 
\end{enumerate}
A probability measure $\sigma$ on the circle is said to have polynomial Fourier decay with parameter $\delta>0$ if there exists $C>0$ such that for all $n\in\Z$,
\[
\left|\hat{\sigma}(n)\right|\leq C|n|^{-\delta}. 
\]
A standard exercise on differentiation by parts shows that if $\sigma$ is Lebesgue absolutely continuous with a $C^1$ density function then it has a polynomial Fourier decay of order 1. Kauffman showed in \cite{Kaufman_1980} that there exists singular measures with polynomial decay. See also \cite{MR3466857, 2023arXiv230601275A,MR4887758} and the references therein for recent results on polynomial Fourier decay for singular measures. 

\begin{thm}\label{thm: Gaussian no rec}
Let $\sigma$ be a measure on the circle with polynomial decay with parameter $\delta>0$, $\left(X,\mathcal{B},m,T\right)$ a Gaussian automorphism with spectral measure $\sigma$ and $k\in\mathbb{N}$ such that $2k\delta>1$. There exists an $m$ preserving transformation $S:X\to X$, isomorphic to $T$, and  $B\in\B^{\otimes k}$ such that for all $n\in\Z$, 
\[
B\cap \left(T^{\otimes k}\right)^{-n}B\cap \left(S^{\otimes k}\right)^{-n}B = \emptyset. 
\]
\end{thm}
\begin{proof}
Let $\mathcal{H}$ be the $L^2$ closure of $\mathrm{span}\left\{f\circ T^n:n\in\Z\right\}$ where $f\in L^2(X,m)$ is a function whose times series is a Gaussian process with spectral measure $\sigma$ which generates $\mathcal{B}$. 

Let  $\mathcal{K}$ be the linear subspace in $\mathcal{H}$ generated by $f$ and denote by $P:\mathcal{H}\to \mathcal{H}$ the projection operator to $\mathcal{K}$. The linear operator $W$, defined by
\[
Wg:=P g-(I-P)g,
\]
is a unitary operator on $\mathcal{H}$ and $W=W^{-1}$. It follows that the operator $V=WUW:\mathcal{H}\to\mathcal{H}$ is a unitary operator which is unitarily equivalent to $U$, the Koopman operator of $T$. Let $S:X\to X$ be a measure preserving transformation, isomorphic to $T$, such that the Koopman operator of $S$ is $V$. With this we have defined the pair of transformations $T,S$. 

It remains to find the set $B$ where double recurrence fails for the pair $T^{\otimes k},S^{\otimes k}$. To that end, first consider the set $A:=\left\{f>1\right\}\in\mathcal{B}$. Since $Pf=f$ and $P^2=P$, for all $n\in \N$,
\begin{equation}\label{eq: Proj of U and V}
PV^nf=PU^nf,    
\end{equation}
and by the definition of $W$,
\[
(I-P)V^nf=-(I-P)U^nf.
\]
The latter implies that for each $n\in \Z$, either $(I-P)V^nf\leq 0$ or $(I-P)U^nf\leq 0$. Hence we deduce that,
\begin{align*}
T^{-n}A\cap S^{-n}A&=\left\{f\circ T^n>1,f\circ S^n>1\right\}\\
&=\left\{U^nf>1,V^nf>1\right\}\\
&\subset\left\{V^nf>1,(I-P)V^nf\leq 0\right\}\cup \left\{U^nf>1,(I-P)U^nf\leq 0\right\}\\
&\subset \left\{PV^nf>1,(I-P)V^nf\leq 0\right\}\cup \left\{PU^nf>1,(I-P)U^nf\leq 0\right\}.
\end{align*}

Taking into consideration \eqref{eq: Proj of U and V}, we conclude that 
\[
T^{-n}A\cap S^{-n}A\subset \left\{PU^nf>1\right\}.
\]

By Markov's inequality,
\begin{align*}
m\left(T^{-n}A\cap S^{-n}A\right)&\leq m\left(PU^nf>1\right)\\&\leq \|PU^n f\|_2^2\ \ \ \ \ \ \ \ \ \ \ \text{as}\ P\ \text{is a rank one projection to}\ \mathcal{K},\\
&=\left|\left<f,U^nf\right>\right|^2\cdot \|f\|_2^2\\
&=\left|\hat{\sigma}(n)\right|^2\cdot \sigma\left(S^1\right)\leq Cn^{-2\delta}. 
\end{align*}

This shows that writing $D:=A^{\otimes k}\in\B^{\otimes k}$,
\[
\sum_{n=1}^\infty m^{\otimes k}\left(D\cap \left(T^{\otimes k}\right)^{-n}D\cap \left(S^{\otimes k}\right)^{-n}D\right)\lesssim \sum_{n=1}^\infty n^{-2k\delta}<\infty.
\]
Arguing in a similar way as in \cref{sub: c=d=1}, one finds a set $B\subset D$, with $m^{\otimes k}(B)>0$, such that for all $n\in \N$,
\[
B\cap \left(T^{\otimes k}\right)^{-n}B\cap \left(S^{\otimes k}\right)^{-n}B=\emptyset.
\]
\end{proof}
Using this result we can now prove \cref{thm:recurrence_2}. 

\begin{proof}[Proof of \cref{thm:recurrence_2}]
Let $c,d\in\Z\setminus \{0\}$. Choose $\sigma$ a singular measure on the circle which has polynomial decay of order $\delta$ (cf. the measure in \cite{Kaufman_1980}) and $k\in\N$ such that $2k\delta>1$. By \cref{thm: Gaussian no rec}, there exist Gaussian automorphisms $P,Q$ of a probability space $\left(Y,\mathcal{C},m\right)$ with spectral measure $\sigma$ and a set $A\in\mathcal{C}^{\otimes k}$ such that $m^{\otimes k}(A)>0$ and for all $n\in\N$,
\[
A\cap \left(P^{\otimes k}\right)^{-n}A\cap \left(Q^{\otimes k}\right)^{-n}A = \emptyset. 
\] We set $(X,\B,\mu)$ to be the probability space $(Y^{\otimes k},\mathcal{C}^{\otimes k},m^{\otimes k})$. Since $\sigma$ is a singular Rajchman measure, both $P$ and $Q$ are mixing, zero entropy transformations \cite{MR191718,MR1231420}.  As every Gaussian automorphism has roots of all orders, there exists $L,R$ Gaussian automorphisms of $(Y,\B,m)$ such that $L^c=P$ and $R^d=Q$. It follows that $L$ and $R$ (and hence $L^{\otimes k}$ and $R^{\otimes k}$) are mixing, zero entropy transformations. It remains to note that for $A\in\mathcal{C}^{\otimes k}$ as above (with $m^{\otimes k}(A)>0$) and for all $n\in\N$,
\[
A\cap \left(L^{\otimes k}\right)^{-cn}A\cap \left(R^{\otimes k}\right)^{-dn}A=A\cap \left(P^{\otimes k}\right)^{-n}A\cap \left(Q^{\otimes k}\right)^{-n}A = \emptyset. 
\] Finally setting $T = L^{\otimes k}$ and $S = R^{\otimes k}$, we obtain the statement of the Theorem. 
\end{proof}

\appendix
\section{}\label{App_1}
The results on the failure of double recurrence for non-commuting transformations make use of the mixing property of skew product extensions. The following result (see \Cref{prop:WM}), the proof for which is standard and well known, is used in \cref{prop:TT_WM_0}. For the sake of completeness, in this appendix, we provide a proof for \Cref{prop:WM}. 

Given a probability-preserving system $(X,\B,\mu,T)$, a measurable function $f:X\to \mathbb{Z}^d$, for $d\geq 1$, we denote by $T_f:X\times \{0,1\}^{\Z^d}\to X\times \{0,1\}^{\Z^d}$, the skew product extension of $T$ by $f$ defined as following
\[
T_f(x,\omega)=(Tx,\sigma_{f(x)}\omega).
\]
Here $\sigma$ is the ($\Z^d$) shift action of $\{0,1\}^{\Z^d}$. For every $\eta$, a stationary product measure on $\{0,1\}^{\Z^d}$, $T_f$ preserves $\mu\times \eta$. We choose $\eta$ to be the product measure on $\{0,1\}^{\Z^d}$ with marginals $\frac{1}{2}\left(\delta_{0}+\delta_1\right)$. In order not to burden the notations, we set $m :=\mu\times \eta$. 
\begin{prop}\label{prop:WM}
Let $(X,\B,\mu,T)$ be a probability-preserving system and let $f:X\to \mathbb{Z}^d$ be a measurable function. Assume that for all $M>0$, 
\[
\lim_{n\to\infty}\mu\left(\left\|S_n(f)\right\|_\infty\leq M\right)=0.
\]
Then $T$ is a  weak mixing (respectively mixing) if and only if $T_f$ is a weak mixing (respectively mixing). 
\end{prop}
\begin{proof}
The projection $\pi: X\times\{0,1\}^{\Z^d}\to X$ is a factor map from $T_f$ to $T$. As a factor of a weak mixing (respectively mixing) transformation is always weak mixing (respectively mixing), it follows that if $T_f$ is weak mixing (respectively mixing) then so is $T$. It remains to show the implications from $T$ to $T_f$. 

Assume first that $T$ is mixing. In order to show that $T_f$ is mixing it suffices to show that for every $A_1,A_2\in\B$ and $B_1,B_2\subset \{0,1\}^{\Z^d}$ a finite union of cylinder sets,
\[
\lim_{n\to\infty}m\left(A_1\times B_1\cap T_f^{-n}\left(A_2\times B_2\right)\right)=\left(\mu\left(A_1\right)\eta\left(B_1\right)\right)\left(\mu\left(A_2\right)\eta\left(B_2\right)\right). 
\]
Let $M\in\N$ be large enough so that $1_{B_1}$ and $1_{B_2}$ are functions of $\{0,1\}^{[-M,M]^d}$.  Note that for all $m\in\Z^d$ with $\|m\|_\infty>2M$, $m+[-M,M]^d$ is disjoint from $[-M,M]^d$ and thus $B_1$ and $\sigma_{-m}B_2$ are independent (under $\eta$). For $n\in\N$, write 
\[
E_n:=\left\{x\in X:\ \left\|S_n(f)(x)\right\|_\infty\leq 2M\right\}\in\B. 
\]
By Fubini's theorem,
\begin{align*}
m\left(A_1\times B_1\cap T_f^{-n}\left(A_2\times B_2\right)\right)&=\int 1_{A_1}1_{A_2}\circ T^n\left(\int 1_{B_1}1_{B_2}\circ \sigma_{S_n(f)}d\eta\right)d\mu\\
&=I(n)+II(n),
\end{align*}
where 
\begin{align*}
I(n)&= \int_{X\setminus E_n} 1_{A_1}1_{A_2}\circ T^n\left(\int 1_{B_1}1_{B_2}\circ \sigma_{S_n(f)}d\eta\right)d\mu,\ \  \text{and}\\
II(n)&= \int_{E_n} 1_{A_1}1_{A_2}\circ T^n\left(\int 1_{B_1}1_{B_2}\circ \sigma_{S_n(f)}d\eta\right)d\mu.
\end{align*}
By the assumption on $f$, $\lim_{n\to \infty}\mu\left(E_n\right)=0$. We deduce from this and
\[
0\leq II(n)\leq \mu\left(E_n\right),
\]
that $\lim_{n\to\infty}II(n)=0$. For all $x\in X\setminus{E_n}$, $\left\|S_n(f)(x)\right\|_\infty\geq 2M$ and
\[
\int 1_{B_1}1_{B_2}\circ \sigma_{S_n(f)(x)}d\eta=\eta\left(B_1\right)\eta\left(\sigma_{-S_n(f)(x)}B_2\right)=\eta\left(B_1\right)\eta\left(B_2\right). 
\]
Consequently, 
\begin{equation}\label{eq: T mixing to T_f mixing}
I(n)= \eta\left(B_1\right)\eta\left(B_2\right)\int_{X\setminus E_n} 1_{A_1}1_{A_2}\circ T^n d\mu. 
\end{equation}
In addition, 
\[
\mu\left(A_1\cap T^{-n}A_2\right)-\mu\left(E_n\right)\leq \int_{X\setminus E_n} 1_{A_1}1_{A_2}\circ T^n d\mu\leq \mu\left(A_1\cap T^{-n}A_2\right). 
\]
By the mixing property of $T$ and $\lim_{n\to\infty}\mu\left(E_n\right)=0$, we deduce that,
\[
\lim_{n\to \infty}\int_{X\setminus E_n} 1_{A_1}1_{A_2}\circ T^n d\mu=\mu\left(A_1\right)\mu\left(A_2\right),
\]
and as a consequence (recall \eqref{eq: T mixing to T_f mixing}),
\begin{align*}
\lim_{n\to\infty}m\left(A_1\times B_1\cap T_f^{-n}\left(A_2\times B_2\right)\right)&=\lim_{n\to\infty}\left(I(n)+II(n)\right) \\
&=\left(\mu\left(A_1\right)\eta\left(B_1\right)\right)\left(\mu\left(A_2\right)\eta\left(B_2\right)\right)\\
&=m\left(A_1\times B_1\right) m\left(A_2\times B_2\right)
\end{align*} as needed. This shows that $T_f$ is mixing for a dense collection of sets in $\B_{X\times\{0,1\}^{\Z^d}}$ (the Borel-$\sigma$ algebra on the product space $X\times\{0,1\}^{\Z^d}$). It follows from this that $T_f$ is mixing.

Now we assume $T$ is weak mixing and show that this implies the weak mixing property of $T_f$. As before let $A_1,A_2\in\B$ and $B_1,B_2\subset\{0,1\}^{\Z^d}$ two sets which are a finite union of cylinder sets. Since $T$ is weak mixing, there exists $J\subset \N$ of Banach density $1$ such that 
\[
\lim_{J\ni n\to\infty}\mu\left(A_1\cap T^{-n}A_2\right)=\mu\left(A_1\right)\mu\left(A_2\right). 
\]
Arguing as before (in the mixing $T$ case) we see that,
\[
\lim_{J\ni n\to\infty}m\left(\left(A_1\times B_1\right)\cap T_f^{-n}\left(A_2\times B_2\right)\right)=m\left(A_1\times B_1\right) m\left(A_2\times B_2\right). 
\]
Since $J$ is of Banach density $1$, 
\[
\lim_{N\to\infty}\frac{1}{N}\sum_{n=1}^N\left|m\left(\left(A_1\times B_1\right)\cap T_f^{-n}\left(A_2\times B_2\right)\right)-m\left(A_1\times B_1\right)m\left(A_2\times B_2\right)\right|=0.
\]
As before this shows that $T_f$ is weak mixing for a dense collection of sets in $\B_{X\times\{0,1\}^{\Z^d}}$. The weak mixing property of $T_f$ follows from this.

\end{proof}
\bibliographystyle{alpha}
\bibliography{lib}

\end{document}